\documentclass[11pt]{amsart}

\usepackage[utf8]{inputenc}
\usepackage{amsfonts}
\usepackage{amsmath}
\usepackage{amssymb}
\usepackage{breqn}
\usepackage{mathtools}
\usepackage{amsthm}
\usepackage[utf8]{inputenc}
\usepackage[T1]{fontenc}
\usepackage{mathrsfs}
\theoremstyle{definition}
\newtheorem{theorem}{Theorem}[section]
\newtheorem{lemma}[theorem]{Lemma}
\newtheorem{remark}[theorem]{Remark}
\newtheorem{definition}[theorem]{Definition}
\newtheorem{corollary}[theorem]{Corollary}
\newtheorem{example}[theorem]{Example}
\newtheorem{proposition}{Proposition}[section]

\newcommand{\cd}{\mathrm{cd}}

\title{Relative cohomological dimension of a relatively hyperbolic pair}
\author{Harsh Patil}
\date{\vspace{-5ex}}

\begin{document}
\maketitle
\begin{abstract}
    We show that the relative cohomological dimension $\cd_R(G,H)$ of a relatively hyperbolic pair $(G,H)$ is always finite when $G$ does not contain $R$-torsion. We also show that this dimension is preserved under quasi-isometries, provided that $G$ is torsion-free and the peripheral subgroup $H$ is unconstricted and of type $F_{\infty}$.  
    As a corollary of our methods, we compute $\cd_{\mathbb{Z}}(G,H)$ in several cases.
\end{abstract}

\section{Introduction}
A relatively hyperbolic group is a finitely generated group equipped with a word metric which satisfies properties similar to the fundamental group of a cusped hyperbolic manifold. Relatively hyperbolic groups were introduced by Gromov in his monograph on hyperbolic groups \cite{Gromov}. Farb [Far94], in his PhD thesis, developed the theory of relative
hyperbolicity using an alternative definition, which is the one we work
with here. Farb’s definition (definition 2.5) is equivalent to that of
Gromov. Alternative characterizations were given by Bowditch \cite{Bowditch} and Yaman \cite{asliyaman}.
Examples of relatively hyperbolic groups include fundamental groups of cusped hyperbolic manifolds, non-cocompact lattices in Lie groups of real rank 1 and geometrically finite convergence groups acting on non-empty perfect
compact metric spaces \cite{asliyaman}.
The problem of classifying relatively hyperbolic groups up to quasi-isometry has attracted significant attention. By a result of Schwartz \cite{Schwartz}, non-uniform lattices of isometries of a rank-one symmetric space are known to be quasi-isometrically rigid: if a group $G$ is quasi-isometric to a non-uniform lattice in a rank-one Lie group $\Gamma$, then $G$ is a finite extension of a non-uniform lattice in $\Gamma$. 
Drutu and Sapir \cite{DRUTU2005959} proved that if a group $G$ is relatively hyperbolic with respect to a finite collection of unconstricted subgroups (cf. below Theorem \ref{qi_theorem} for the definition of an unconstricted group) and $G'$ is another group quasi-isometric to $G$, then $G'$ is also relatively hyperbolic. Furthermore, every peripheral subgroup of $G'$ is quasi-isometric to some peripheral subgroup of $G$. Thus, the property of being relatively hyperbolic with respect to a fixed set of unconstricted peripheral subgroups is quasi-isometrically invariant. Groff \cite{groff} builds on this result to show that if two relatively hyperbolic groups are quasi-isometric, then the corresponding Bowditch boundaries are homeomorphic. 

In this article, we study the relative cohomology $H^{i}(G,H;M)$ of a relatively hyperbolic group $G$ relative to its peripheral subgroup $H$ with coefficients in a $G$-module $M$. Our primary motivation behind this is the following result by Manning and Wang \cite{Manning2018CohomologyAT}, which relates the Bowditch boundary $\partial_{B}(G,H)$ of $G$ with respect to $H$ with the relative cohomology $H^{i}(G,H;RG)$:

\begin{theorem}\label{manningwang}
    Let $(G,H)$ be a relatively hyperbolic pair, such that both $G$ and $H$ are of type $F_{\infty}$. Let $R$ be any ring. Then, for every $i$, there is an isomorphism of $RG$-modules
    $$H^{i}(G,H;RG)\cong 
    \check{H}^{i-1}(\partial_{B}(G,H);R). $$
\end{theorem}

The right-hand side, $\check{H}^{i-1}(\partial_{B}(G,H);R)$, denotes the \v{C}ech cohomology of the Bowditch boundary (with coefficients in $R$), which is a topological invariant. Thus, Theorem \ref{manningwang} provides an algebraic approach to understanding the topology of the Bowditch boundary. 
The main result of this article is as follows (see Section \ref{section_cohomology of_a_pair} for the definition of the $RG$-module $\Delta_{G/H}^{R}$):

\begin{theorem}\label{maintheorem1} 
Let $G$ be a group and let $H < G$ such that $G$ is hyperbolic relative to $H$ and such that $G$ does not contain $R$-torsion. Then the $RG$-module $\Delta^{R}_{G/H}$ admits a finite-length projective resolution of $RG$-modules in which each non-trivial term is finitely generated; i.e., it is a module of type $FP$ over $RG$. Consequently, $\textrm{cd}_{R}(G, H)$ is finite.
\end{theorem}
Recall that, for a given ring $R$ and a group $G$, $G$ is said to have \emph{no $R$-torsion} if the order of any finite subgroup of $G$ is invertible in $R$. 
\begin{remark}
    The relative cohomological dimension was already known to be finite under the  assumption that the peripherals are all of type $FP$ (cf. \cite[Theorem 5.1]{Manning2018CohomologyAT}). We relax this condition in Theorem \ref{maintheorem1} and only demand that $G$ contain no $R$-torsion. 
\end{remark}

The following corollary is an easy consequence of Theorem \ref{maintheorem1}:
\begin{corollary}\label{cd_corollary}
    Let $G$ be a torsion-free group such that $G$ is hyperbolic relative to a subgroup $H$. Then, $G$ has finite cohomological dimension if and only if $H$ does. 
\end{corollary}

We also prove the following corollary, which will be used in the proof of Theorem \ref{qi_theorem}: 
\begin{corollary}\label{main_corollary}
Let $G$ be a torsion-free group and $H<G$ be a subgroup such that $G$ is relatively hyperbolic with respect to $H$. Then,
    $\cd_{R}(G,H)=\sup\{i|H^{i}(G,H;RG)\neq 0\}$.
\end{corollary}
\begin{remark}
    Again, the conclusion of Corollary \ref{main_corollary} was already known to be true under the assumption that all peripheral subgroups are of type $FP$ \cite[Prop. 2.24]{Manning2018CohomologyAT}. We show that  the conclusion holds for a much larger collection of relatively hyperbolic pairs. 
\end{remark}

Finally, we prove the following theorem that allows us to distinguish certain relatively hyperbolic groups up to quasi-isometry:
\begin{theorem}\label{qi_theorem}
    Let $G$ and $G'$ be two quasi-isometric groups that are relatively hyperbolic with respect to subgroups $H<G$ and $H'<G'$, respectively. Assume further that neither $G$ nor $G'$ contain $R$-torsion. Assume that both $H$ and $H'$ are unconstricted and of type $F_{\infty}$. Then, $\cd_{R}(G,H)=\cd_{R}(G',H')$. 
\end{theorem}
A group is said to be \textit{unconstricted} if one of its asymptotic cones is without cut-points. Any group that satisfies a law is unconstricted. For more examples cf. \cite[pg. 6]{DRUTU2005959}. 

\begin{remark}
  The cohomological dimension $\cd_{R}(G,H)$ is not an entirely new quasi-isometry invariant for relatively hyperbolic groups. Rather, its quasi-isometry invariance follows from the fact that the Bowditch boundaries of two quasi-isometric pairs are homeomorphic, in combination with Theorem \ref{manningwang}. While describing the Bowditch boundary entirely or calculating the relative cohomology may seem difficult, we demonstrate with examples that $\cd_{\mathbb{Z}}(G,H)$ can be computed explicitly in some cases.
\end{remark}
Throughout the paper, we restrict ourselves to the case of a single peripheral subgroup. The analogous statements for multiple peripheral subgroups are stated in Section 6.  
\section{Preliminaries}
 A \textit{group pair} consists of a group $G$ and a subgroup $H<G$. It is denoted by $(G,H)$. 
\subsection{Relative Hyperbolicity}
We recall some standard notions from the theory of relatively hyperbolic groups.
The interested reader can refer to \cite{farb1998relatively}, \cite{Bowditch}, and \cite{osin2006relatively}. 
\begin{definition}
Let $G$ be a finitely generated group and $H$ be a subgroup of $G$. Let $S$ be some finite generating set of $G$. Let $\Gamma(G,S)$ denote the Cayley graph of $G$ with respect to $S$. Assigning each edge unit length, the graph $\Gamma(G,S)$ becomes a metric space equipped with the path metric. 
\par We form a new graph $\hat{\Gamma}(G,S,H)$ as follows: Start with the Cayley graph $\Gamma(G,S)$ and for each coset $gH$ of $H$ add a vertex $v(gH)$ and an edge of length $\frac{1}{2}$ between $v(gH)$ and each element $gh$ of $gH$.  
\par This new graph is called the \emph{coned-off Cayley graph} of the pair $(G,H)$ with respect to the generating set $S$.  $\hat{\Gamma}(G,S,H)$ is equipped with the path metric, which we denote by $\hat{d}$. 
\end{definition}

\begin{definition}A group $G$ is said to be \emph{weakly relatively hyperbolic} with respect to a subgroup $H$ if the coned-off Cayley graph $\hat{\Gamma}(G,S,H)$ is Gromov hyperbolic. 
\end{definition}

\begin{definition} Let $G$ be a group and $H$ be a subgroup. 
   Let $c\colon [0,l]\rightarrow \Gamma(G,S)$ be a path in the Cayley graph of $G$. Its image $\hat{c}$ in the coned-off graph $\hat{\Gamma}(G,S,H)$ is constructed by replacing every maximal subpath $c|_{[a,b]}$ all of whose vertices lie in the same coset $gH$ with a path of length $1$ from $c(a)$ to $c(b)$ passing through the cone point $v(gH)$. Let $l'$ denote the length of $\hat{c}$. $c$ is said to be a \emph{relative geodesic} if $\hat{c}$ is a geodesic in $\hat{\Gamma}(G,S,H)$. The length of $\hat{c}$ is said to be the \emph{relative length} of $c$. Let $T>1$. The path $c$ is called a $T$-relative quasi-geodesic if $\hat{c}$ is a $T$-quasi-geodesic, i.e.,
   $$T^{-1}|t-t'|\leq \hat{d}(\hat{c}(t),\hat{c}(t'))\leq T|t-t'| \quad \forall t,t'  \in [0,l'] $$
   A path $c$ is said to be a \emph{relative quasi-geodesic} if it is a $T$-relative quasi-geodesic for some $T>0$.
\end{definition}

A path $c\colon [a,b]\rightarrow \Gamma(G,S)$ is said to \textit{travel in a coset $\gamma H$ for less than $r$} if for any maximal subsegment of $c$ all of whose vertices lie in $\gamma H$, the endpoints are at a distance less than $r$. A path $c$ is said to \emph{travel in a coset $\gamma H$ for more than $r$} if there exists a subsegment of $c$ with vertices in $\gamma H$ such that its endpoints are at least at a distance $r$ from each other.

\begin{definition}[Bounded Coset Penetration]
  Let $G$ be a group and $H$ be a subgroup. The pair $(G,H)$ is said to satisfy the BCP (Bounded Coset Penetration) property if for every $T>0$ there exists an $r=r_{BCP}(T)$ such that for every pair of $T$-relative quasigeodesics $c_1$, $c_2$ starting and ending at the same point in $\Gamma(G,S)$ satisfy the following two conditions: 
  \begin{enumerate}
      \item If $c_1$ travels more than $r$ in a coset, then $c_2$ enters the same coset.
      \item If $c_1$ and $c_2$ enter the same coset, then the entry points (respectively, exit points) of $c_1$ and $c_2$ are at most $r$ distance apart. 
  \end{enumerate}
\end{definition}

\begin{definition}[Relative Hyperbolicity]
A group is said to be relatively hyperbolic with respect to a subgroup $H$ if it is weakly relatively hyperbolic and satisfies the BCP property. The pair $(G,H)$ is then said to be a \textit{relatively hyperbolic pair}. 
\end{definition}
\subsection{Quasi-isometries of Groups and Pairs}
Let $G_{1}$ and $G_{2}$ be two finitely generated groups and let $H_{1}$ and $H_{2}$ be subgroups of $G_{1}$ and $G_{2}$, respectively. Let $S_{1}\subset G_{1}$ and $S_{2}\subset G_{2}$ be finite symmetric generating sets of $G_{1}$ and $G_{2}$, respectively. Let $d_{S_{1}}$ and $d_{S_{2}}$ denote the respective word metrics on $G_{1}$ and $G_{2}$. 

\begin{definition}[Quasi-isometry]
    A map $q \colon G_{1} \rightarrow G_{2}$ is said to be a quasi-isometry if 
    \begin{enumerate}
        \item there exist constants $A, B > 0$ such that 
        $$
        A^{-1} d_{S_{1}}(x, y) - B \leq d_{S_{2}}(q(x), q(y)) \leq A d_{S_{1}}(x, y) + B
        $$
        for all $x, y \in G_{1}$, and
        \item there exists a constant $C > 0$ such that for every $y \in G_{2}$, there exists $x \in G_{1}$ such that $d_{S_{2}}(q(x), y) \leq C$.
    \end{enumerate}
\end{definition}

Following \cite{Hughes_Martínez-Pedroza_2023}, we define quasi-isometries of pairs.

\begin{definition}[Quasi-isometry of pairs]
    Let $q \colon G_{1} \rightarrow G_{2}$ be a quasi-isometry between groups $G_{1}$ and $G_{2}$. Suppose, in addition, that there exist subgroups $H_{1} < G_{1}$ and $H_{2} < G_{2}$, and a constant $M > 0$ such that the following two conditions are satisfied: 
    \begin{enumerate}
        \item For any left coset $gH_{1}$ of $H_{1}$, the set 
        $$
        \{g' \in G_{2} \mid d^{\mathrm{Haus}}_{S_{2}}(q(gH_{1}), g'H_{2}) < M \}
        $$
        is non-empty.
        \item For any left coset $g'H_{2}$ of $H_{2}$, the set 
        $$
        \{g \in G_{1} \mid d^{\mathrm{Haus}}_{S_{2}}(q(gH_{1}), g'H_{2}) < M \}
        $$
        is non-empty.
    \end{enumerate}
    Then, we say that $q$ is a quasi-isometry between the pairs $(G_{1}, H_{1})$ and $(G_{2}, H_{2})$.
\end{definition}

In Section \ref{section_4}, we show that a quasi-isometry between two groups $G_{1}$ and $G_{2}$, which are relatively hyperbolic with respect to unconstricted subgroups $H_{1}$ and $H_{2}$, respectively, is automatically a quasi-isometry of pairs $(G_{1}, H_{1})$ and $(G_{2}, H_{2})$.
\subsection{Bowditch Boundary}
Let $(G,H)$ be a relatively hyperbolic pair. Then, there exists a non-empty, perfect, metrizable compactum $M$ such that $G$ admits a convergence group action on $M$, and the action satisfies three additional properties:  $M$ consists only of conical limit points and bounded parabolic points, there is exactly one orbit of bounded parabolic points, and
   the stabilizer of any bounded parabolic point is a conjugate of $H$.
Furthermore, if $M'$ is any other compactum that satisfies the above properties, then $M$ and $M'$ are $G$-equivariantly homeomorphic. Such a space is called the Bowditch boundary of the pair $(G,H)$ and is denoted by $\partial_{B}(G,H)$. We refer the reader to \cite{Manning2018CohomologyAT} for an explicit description of $\partial_{B}(G,H)$ and to \cite{asliyaman} for a verification of the properties stated above. In this article, we do not work directly with the Bowditch boundary. The two facts regarding the Bowditch boundary that we will use in this paper are Theorem \ref{manningwang} and the following theorem, which is a consequence of a result by Groff \cite[Theorem 6.3, Corollary 6.5]{groff}:

\begin{theorem}\label{groff}
Let $G_{1}$ and $G_{2}$ be finitely generated groups such that $G_{1}$ is relatively hyperbolic with respect to a subgroup $H_{1}$, and $G_{2}$ is relatively hyperbolic with respect to a subgroup $H_{2}$. Let $q \colon (G_{1},H_{1}) \rightarrow (G_{2},H_{2})$ be a quasi-isometry of pairs. Then, the Bowditch boundaries $\partial_{B}(G_{1},H_{1})$ and $\partial_{B}(G_{2},H_{2})$ are homeomorphic. 
\end{theorem}

\begin{remark}
    The above theorem is just a special case of the more general result by Groff \cite{groff}. Groff assumes that $G_{1}$ is relatively hyperbolic with respect to a finite collection $\mathcal{A}_{1}$ of peripheral subgroups and does not assume relative hyperbolicity for $G_{2}$. Rather, he uses a result of Behrstock-Drutu-Mosher \cite[Theorem 4.8]{behrstockdrutumosher} to infer the existence of a collection $\mathcal{A}_{2}$ of subgroups of $G_{2}$ with respect to which $G_{2}$ is relatively hyperbolic. Additionally, every subgroup $H \in \mathcal{A}_{1}$ is quasi-isometric to some group $H' \in \mathcal{A}_{2}$, and for any given $g \in G$ and $H \in \mathcal{A}_{1}$, the image of $q(gH)$ lies within a bounded Hausdorff distance of some coset $g'H'$, where $g' \in G'$ and $H' \in \mathcal{A}_{2}$. Groff then uses these two facts in order to establish the quasi-isometry between the cusped spaces $X(G_{1}, \mathcal{A}_1)$ and $X(G_{2}, \mathcal{A}_{2})$. In general, the number of elements in $\mathcal{A}_{1}$ and $\mathcal{A}_{2}$ may not necessarily be equal. As we would like to restrict ourselves to the class of groups that are relatively hyperbolic with respect to a single peripheral subgroup, we assume that the quasi-isometry $q$ is a quasi-isometry of pairs $(G,H)$ and $(G',H')$ rather than relying on \cite[Theorem 4.8]{behrstockdrutumosher}.
\end{remark}

\subsection{Cohomology of Pairs}\label{section_cohomology of_a_pair}

Given a pair $(G,H)$, we recall the definition of the relative cohomology of $(G,H)$ with coefficients in a $RG$-module $M$, following \cite{BIERI1978277}.

Let $\mathbb{Z}[G/H]$ be the $\mathbb{Z}G$-module generated by the set $G/H$ of left cosets of $H$. Let $\epsilon\colon \mathbb{Z}[G/H] \rightarrow \mathbb{Z}$ be the $\mathbb{Z}G$-module morphism that sends every generator to $1$. Define $\Delta_{G/H} \coloneqq \ker \epsilon$. For an arbitrary ring $R$ the $RG$ module $\Delta_{G/H}^{R}$ is defined as $\Delta_{G/H} \otimes R$.
  Let $M$ be an $RG$-module. Then,
\[
H^{i}(G,H;M) \coloneqq \mathrm{Ext}_{RG}^{i-1}(\Delta_{G/H},M).
\]

\begin{remark}
We recall what this means. Let 
\[
\dots \rightarrow F_{2} \rightarrow F_{1} \rightarrow F_{0} \rightarrow \Delta_{G/H}^{R} \rightarrow 0
\]
be a projective resolution of the $G$-module $\Delta_{G/H}^{R}$. Then, $\mathrm{Ext}^{\ast}_{RG}(\Delta_{G/H}^{R},M)$ denotes the cohomology of the cochain complex
\[
\dots \leftarrow \mathrm{Hom}_{RG}(F_{1},M) \leftarrow \mathrm{Hom}_{RG}(F_{0},M) \leftarrow 0.
\]
\end{remark}

\begin{definition}[Cohomological Dimension of a Pair]
Let $G$ be a group, and let $H < G$ be a subgroup. The cohomological dimension of the pair $(G,H)$ with coefficients in a ring $R$, denoted by $\cd_R(G,H)$, is defined as
\[
\cd_R(G,H) \coloneqq \sup \{ i \mid H^{i}(G,H;M) \neq 0 \text{ for some } G\text{-module } M \}.
\]
\end{definition}

\begin{proposition}\cite[Prop. 1.1]{BIERI1978277}\label{LES_pair}
For any group pair $(G,H)$ and any $RG$-module $M$, there is a long exact sequence in cohomology:
\[
\dots \rightarrow H^{k}(G;M) \rightarrow H^{k}(H;M) \rightarrow H^{k+1}(G,H;M) \rightarrow H^{k+1}(G;M) \rightarrow \dots
\]
\end{proposition}

We also recall the definition of an Eilenberg–MacLane pair for a group pair $(G,H)$.

\begin{definition}[Eilenberg–MacLane Pair]
Let $G$ be a group and $H$ a subgroup of $G$. Let $(X,Y)$ be a pair of CW-complexes such that $X$ and $Y$ are classifying spaces for $G$ and $H$, respectively, and the inclusion map $i\colon Y \hookrightarrow X$ induces the inclusion $H \hookrightarrow G$ on fundamental groups. Then, the pair $(X,Y)$ is said to be an Eilenberg–MacLane pair for $(G,H)$.
\end{definition}

Explicit constructions of Eilenberg–MacLane pairs for some relatively hyperbolic pairs are given in Section \ref{section_examples}. We now state the following lemma, which is simply \cite[Theorem 1.3]{BIERI1978277} applied to the case of the trivial $\mathbb{Z}G$-module $\mathbb{Z}$:

\begin{lemma}\label{classifyingpairlemma1}
Let $G$ be a group and $H$ a subgroup of $G$. Let $(X,Y)$ be an Eilenberg–MacLane pair for $(G,H)$. Let $H^{i}(X,Y,R)$ denote the cohomology of $X$ relative to $Y$, and let $R$ denote the trivial $RG$-module. Then,
\[
H^{i}(G,H,R) \cong H^{i}(X,Y,R).
\]
\end{lemma}

We will need the following result from \cite[Prop. 2.20]{Manning2018CohomologyAT}:

\begin{lemma}\label{classifyingpairlemma2}
Let $(G,H)$ be a pair, and let $(X,Y)$ be an Eilenberg–MacLane pair for $(G,H)$. Let $\tilde{X}$ denote the universal cover of $X$, and let $\tilde{Y}$ denote the preimage of $Y$ under the canonical projection $ \tilde{X}\rightarrow X$. If $C_{\ast}(\tilde{X},\tilde{Y})$ denotes the relative chain complex, then the cohomology of $\mathrm{Hom}_{G}(C_{\ast}( \tilde{X},\tilde{Y});M)$ is $H^{\ast}(G,H;M)$. If, in addition, $X$ is finite-dimensional, then $\cd_R(G,H) \leq n$, where $n = \dim(X \setminus Y)$.
\end{lemma}

\begin{proof}
The first part of the statement appears in \cite{BIERI1978277}. Now, let $n = \dim(X \setminus Y)$. It follows that $C_{i}(\tilde{X},\tilde{Y}) = 0$ for all $i > n$. Consequently, the cohomology at the $i$-th term of the chain complex $\mathrm{Hom}_{G}(C_{\ast}(\tilde{X}, \tilde{Y});M)$ vanishes for all $i > n$ and for all $RG$-modules $M$. As a result, $\cd_{R}(G,H) \leq n$.
\end{proof}
\subsection{Dehn Fillings of Relatively Hyperbolic Groups}

\begin{theorem}\cite[Theorem 1.1]{Osin2005PeripheralFO}\label{dehn_filling}
Suppose $G$ is hyperbolic relative to $H$. Then, there exists a finite subset $F \subset G \setminus \{e\}$ of nontrivial elements of $G$ with the following property:  
if $N \lhd H$ is a normal subgroup such that $N \cap F = \emptyset$, then the following hold:
\begin{enumerate}
    \item The natural map $q\colon H/N \rightarrow G/\langle\langle N \rangle\rangle$ is injective.
    \item The quotient group $\bar{G} = G/\langle\langle N \rangle\rangle$ is relatively hyperbolic with respect to $\bar{H} = H/N$.
\end{enumerate}
\end{theorem}

We now state the following result due to Sun–Petrosyan \cite[Theorem A.(ii)]{PETROSYANSUN}:

\begin{theorem}\label{binsun}
Let $G$ be relatively hyperbolic with respect to a subgroup $H$. Then, there exists a finite subset $F \subset G \setminus \{e\}$ such that for all normal subgroups $N \lhd H$ satisfying $N \cap F = \emptyset$ and for all $\bar{G}$-modules $A$, there is a natural isomorphism
\[
H^{n}(\bar{G}, \bar{H}; A) \cong H^{n}(G, H; A)
\]
induced by the quotient maps $q\colon G \rightarrow \bar{G}$ and $q\colon H \rightarrow \bar{H}$. As a consequence, 
\[
\cd_{R}(\bar{G}, \bar{H}) \leq \cd_{R}(G, H).
\]
\end{theorem}

We will use the above result in Section~\ref{section_examples} to construct examples of relatively hyperbolic groups with prescribed properties.

\section{A free resolution of finite type for $\Delta_{G/H}^{R}$}
In this section, we provide a proof of Theorem \ref{maintheorem1} and Corollaries \ref{cd_corollary} and \ref{main_corollary}. We start by defining a pair $( {X}_{r,d}, {Y})$ of simplicial complexes, such that $ {X}_{r,d}$ is acyclic and $G$ admits a simplicial action on $ {X}_{r,d}$. Furthermore, $ {Y}$ is a $G$-invariant subcomplex of $ {X}_{r,d}$ and every connected component of $ {Y}$ is contractible. First, we show in Proposition \ref{acyclic}  that $ {X}_{r,d}$ is acyclic. 
 We then show in Proposition 3.2 that the relative chain complex $C_{\ast}( {X}_{r,d}, {Y})$ 
  has finite length and each non-zero term is a finitely generated $RG$-module. Finally, we give a proof of Theorem \ref{maintheorem1}. \newline
We start by defining the complex $ {X}_{r,d}$. The directed system  $\{X_{r,d,s}\}_{s}$ has been defined in \cite{patil2023} and is inspired by the Relative Rips construction given by Dahmani \cite{Dahmani}. 
Let $G$ be relatively hyperbolic with respect to $H<G$. Let $S$ and $T$ be finite generating sets of $G$ and $H$, respectively, such that $T\subseteq S$. Let $d_T$ denote the word metric on $H$ with respect to $T$.
\begin{definition}
Let $G$ and $H$ be as above. Let $r,d>0$. The relative Rips complex $X_{r,d}$ is defined as follows: 
    \begin{enumerate}
    \item $ {X^{(0)}_{r,d}}=G$
    \item There is an edge between two $x,x'\in  {X}_{r,d}^{(0)}$ if one of the two conditions are satisfied:
    \begin{enumerate}
        \item $x$ and $x'$ lie in the same coset of $H$, or
        \item there is a relative geodesic $c\colon [0,l]\rightarrow \Gamma(G,S)$ with endpoints $x,x'$  such that the relative length of $c$ is at most $d$ and it travels less than $3r$ in the first and the last coset and less than $2r$ in any other coset. 
    \end{enumerate}
    \item A collection of $n$ vertices span a $n$-simplex if and only if there is an edge between any two of them in the one skeleton $ {X}_{r,d}^{(1)}$. 
\end{enumerate}
\end{definition}
 The group $G$ acts on the $0$-skeleton of $ {X}_{r,d}$ by left-multiplication. This action extends to a simplicial action on $ {X}_{r,d}$.
   For $r,d,s>0$, consider the flag subcomplex ${X}_{r,d,s}$ of $X_{r,d}$ such that $X_{r,d,s}^{(0)}=X_{r,d}^{(0)}$ and such that 1-skeleton $X_{r,d,s}^{(1)}$ is defined as follows, there is an edge between two $x,x'\in  {X}_{r,d,s}^{(0)}$ if one of the two conditions are satisfied:
    \begin{enumerate}
        \item $x$ and $x'$ lie in the same coset of $H$ and $d_{T}(1,x^{-1}x')\leq s$
        \item there is a relative geodesic $c\colon [0,l]\rightarrow \Gamma(G,S)$ with endpoints $x,x'$  such that the relative length of $c$ is at most $d$ and it travels less than $3r$ in the first and the last coset and less than $2r$ in any other coset.  
    \end{enumerate}

 For every $s>0$, $ {X}_{r,d,s}$ is a $G$-invariant subcomplex of ${X}_{r,d}$. Let ${K}$ denote the subcomplex of $ {X}$ such that $ {K}^{(0)}=H$ and such that there is a $n$-simplex spanning any collection of $(n+1)$ vertices $\{x_{0},x_{1},\dots, x_{n}\}\subset {K}^{(0)}$. In other words, $ {K}$ is the induced subcomplex of $X_{r,d}$ with vertex set $H$. $ {K}$ is the \emph{full simplex } on the set $H$, i.e. for each $n$, every $(n+1)$-element subset of $H$ spans an $n$-simplex. Let $Y$ denote the union of all $G$-translates of all $K$. 
\begin{proposition}\label{acyclic}
Let $G$ be a torsion-free group that is relatively hyperbolic with respect to a subgroup $H$. Let $d>4\delta +2$ and $ r>r_{BCP}(4d)$. Then $ {X}_{r,d}$ is acyclic. 
\end{proposition}

\begin{proof}
Let $\alpha\in H_{n}( {X}_{r,d})$, and let $\sigma$ be an $n$-cycle representing $\alpha$. As $ {X}_{r,d}=\bigcup_{s}  {X}_{r,d,s}$, every simplex in the support of $\sigma$ is contained in $ {X}_{r,d,s}$  for a sufficiently large $s$. It was shown in \cite[Proof of Theorem 1.1]{patil2023} that given an element $\alpha\in H_{n}( {X}_{r,d,s})$, where $d>4\delta +2$ and $ r>r_{BCP}(4d)$, one can express $\alpha$ as a sum $\alpha=\alpha_{1}+\dots+\alpha_{N}$ such that each $\alpha_{i}\in H_{r,d,s}( {X}_{r,d,s})$ is represented by an $n$-cycle that is supported on a single coset of $H$. The subcomplex spanned by each coset of $H$ is simply an infinite-dimensional simplex and is thus contractible. Consequently, $\alpha_i=0$, for each $i$. It follows that $\alpha=0$ i.e. $H_{n}( {X}_{r,d})=0$.  
\end{proof}

\begin{remark}
    It can also be proved that $ {X}_{r,d}$ is simply connected and hence contractible. However, this is not required for the proof of Theorem \ref{maintheorem1}. 
\end{remark}

\begin{proposition}\label{eilenbergmclane}
    Let $( {X}_{r,d}, {Y})$ be as defined above. Then the relative chain complex $C_{\ast}( {X}_{r,d}, {Y})$ has finite length, and every nontrivial term is finitely generated as a $\mathbb{Z}G$-module. 
\end{proposition}

\begin{proof} 
To show that $C_{\ast}( {X}_{r,d}, {Y})$ is a finite chain complex, it suffices to prove that for a sufficiently large natural number $N$, every simplex in $ {X}_{r,d}$ of dimension greater than $N$ lies in $ {Y}$. 

Let $\sigma$ be an $n$-cell in $ {X}$. Suppose $\sigma$ is spanned by $\{x_{0},x_{1},\dots,x_{n}\}$. Then either all $x_i$ lie in a single coset of $H$, or there exist $j, k$ such that $x_{j}$ and $x_{k}$ lie in different cosets $x_{j}H$ and $x_{k}H$ of $H$. In the first case, $ {\sigma}$ is a simplex of $ {Y}$. 

Now suppose that $x_{j}$ and $x_{k}$ lie in different cosets for some $j,k\in\{0,1,\dots,n\}$. Since there exists a path of length at most $3d\cdot r$ that joins $x_{j}$ and $x_{k}$, any other $x_{l}$ is connected to either $x_{j}$ or $x_{k}$ by a path of length at most $3d\cdot r$. In either case, $d_{S}(x_{j},x_{l})\leq 6d\cdot r$. Thus, $x_{l}$ lies in the ball of radius $6d\cdot r$ around $x_{j}$. Thus, any such simplex has dimension at most $|B(x_{j},6d\cdot r)|$. 

It remains to show that each $C_{n}( {X}_{r,d}, {Y})$ is finitely generated as a $\mathbb{Z}G$-module. It suffices to show that the action of $G$ on $ {X}_{r,d}^{(n)}/ {Y}^{(n)}$ has only finitely many $G$-orbits. Let $\sigma \in  {X}_{r,d}^{(n)}/ {Y}^{(n)}$. Up to $G$-translation, we may assume that one of the vertices is the identity element $e\in  {X}_{r,d}^{(0)}$. Since $\sigma$ is not an element of $ {Y}$, there is at least one vertex $v$ of $\sigma$ that does not lie in $H$. Let $v'$ be any other vertex of $\sigma$. Then $d_{S}(e,v')\leq 6d\cdot r$. Thus, there are only finitely many simplices of a fixed dimension that can arise in this way. 
\end{proof}

We are now ready to prove Theorem \ref{maintheorem1}. 
\begin{proof}[Proof of Theorem \ref{maintheorem1}:]
First we prove the statement for $R=\mathbb{Z}$. 
In view of Proposition \ref{eilenbergmclane}, it suffices to show that the $\mathbb{Z}G$-module $H_{1}(X_{r,d},Y)$
is isomorphic to the \(G\)-module \(\Delta_{G/H}\). Indeed, the required projective resolution can then be obtained by appending \(\Delta_{G/H}\) to the chain complex \((C_{\ast}( {X}_{r,d}, {Y}),\partial)\):
\[
\dots \to C_{2}( {X}_{r,d}, {Y}) \to C_{1}( {X}_{r,d}, {Y}) \to H_{1}(X_{r,d},Y)\cong \Delta_{G/H} \to 0.
\]
As $C_0(X_{r,d},Y)=0$ ($X_{r,d}$ and $Y$ have the same vertex set), the map $\partial:C_{2}(X_{r,d},Y)\rightarrow C_1(X_{r,d},Y)$ is surjective, i.e., $C_{1}(X_{r,d},Y)= Z_1(X_{r,d},Y)$. As a result, there is a natural surjective map from $C_{1}(X_{r,d},Y)$ to $H_1(X_{r,d},Y)$.    
By assumption,
there is no torsion, so each $G$-module $C_i(X_{r,d},Y)$ is
projective, in fact free. Moreover, the above complex is exact. For $i\geq 2$, one can deduce that $H_i(X_{r,d},Y)=0$ by using the long exact sequence corresponding to the pair $(X_{r,d},Y)$ and the fact that $X_{r,d}$ and $Y$ both have trivial homology in degrees $\geq 1$.  
Let \((C_{\ast}( {X}_{r,d}),\partial)\) denote the chain complex of \( {X}_{r,d}\). The chain complex of \( {Y}\) forms a subcomplex of \((C_{\ast}( {X}_{r,d}),\partial)\), and the relative chain complex \(C_{\ast}( {X}_{r,d}, {Y})\) is the quotient. Since \( {X}_{r,d}\) is connected, the \(0\)-th homology of \((C_{\ast}( {X}_{r,d}),\partial)\) is isomorphic to the trivial \(G\)-module \(\mathbb{Z}\). The connected components of \( {Y}\) correspond to the cosets \(gH\) of the group \(H\). The \(0\)-th homology \(H_{0}( {Y})\) of \((C_{\ast}( {Y}),\partial)\) is isomorphic to \(\mathbb{Z}[G/H]\). 

The map \( i_{\ast}:H_{0}( {Y})\to H_{0}( {X}_{r,d})\) induced by the inclusion map coincides with the augmentation map \(\epsilon:\mathbb{Z}[G/H]\to \mathbb{Z}\). Consider the first few terms of the long exact sequence in homology associated with the pair \(( {X}_{r,d}, {Y})\):
\[
0=H_{1}(X_{r,d})\rightarrow H_{1}(X_{r,d},Y) \to H_{0}(Y) \xrightarrow{\epsilon} H_{0}( {X}_{r,d}) \to H_{0}(X_{r,d},Y)=0.
\]
Thus, \(H_{1}( {X}_{r,d}, {Y}) \cong \ker(\Delta_{G/H})\). 
Next, we prove the statement for an arbitrary ring $R$:
It suffices to show that the $RG$-module $C_{i}(X_{r,d}, Y; R) := C_{i}(X_{r,d}, Y) \otimes R$ is projective for each $i > 0$. The module $C_{i}(X_{r,d}, Y; R)$ can be expressed as a direct sum $\bigoplus_{\sigma} R\sigma$, where the direct sum ranges over all $i$-simplices $\sigma$ in $X_{r,d}/Y$ modulo $G$. Each summand $R\sigma$ is isomorphic (as an $RG$-module) to $R[G/G_{\sigma}]$, where $G_{\sigma}$ is the stabilizer of $\sigma$ and $G$ acts on the set of left cosets of $G_{\sigma}$ by left multiplication.

Observe that any non-trivial element $g \in G_{\sigma}$ permutes the vertices of $\sigma$ non-trivially. As a result, $G_{\sigma}$ is a finite group of order at most $(i+1)!$. We show that the module $R[G/G_{\sigma}]$ is projective by proving that it is a direct summand of the free $RG$-module $RG$.

Let $q: RG \to R[G/G_{\sigma}]$ denote the canonical surjective map that sends an element $g$ to the corresponding left coset $gG_{\sigma}$, and let $p: R[G/G_{\sigma}] \to RG$ denote the map that sends a coset $gG_{\sigma}$ to the element $|G_{\sigma}|^{-1} \sum_{h \in G_{\sigma}} gh$. Since we assumed $G$ does not contain
$R$-torsion, $|G_\sigma|$ is invertible in $R$ and $p$ is well-defined. It is easy to see that $q \circ p$ is the identity map on $R[G/G_{\sigma}]$. Thus, $R[G/G_{\sigma}]$ is projective.
\end{proof}
\begin{remark}
Other Rips-style constructions for relatively hyperbolic groups have been defined in the literature.  
The authors of \cite{Martínez-Pedroza_Przytycki_2019} construct, for a given relatively hyperbolic pair \((G,H)\), a finite-dimensional contractible simplicial complex \(\Gamma_{n}^{\triangle}\) such that \(G\) acts cocompactly on \(\Gamma_{n}^{\triangle}\), the stabilizer of each vertex is either finite or conjugate to \(H\), and the stabilizer of each simplex of dimension at least one is finite (in particular, trivial if \(G\) is torsion-free). 
$\Gamma_{n}^{\triangle}$ is constructed as follows: \newline
Let $S\subset G$ and be a finite symmetric generating set of $G$ let $d_{S}$ denote the word metric on $G$. The 0-skeleton ${\Gamma_{n}^{\triangle}}^{(0)}$ is the same as that of the coned-off Cayley graph, namely, $ G\sqcup \{v(gH)\}_{gH\in G/H}$. Let $d:{\Gamma_{n}^{\triangle}}^{(0)}\rightarrow \mathbb{R}$ such that $d(g_{1},g_{2})=d_{S}(g_{1},g_{2})$ for all $g_{1},g_{2}\in G$ and $d_{S}(g_{1},v(g_{2}H))=d_{S}(g_{1},g_{2}H)$ and $d(g_{1}H,g_{2}H)$ is equal to the distance between the two cosets when considered as subsets of the metric space $(G,d_{S})$.  Note that  $d_S$ is not
a metric.
Two vertices $x,y$ are declared to be adjacent if $d(x,y)<n$. $\Gamma_{n}^{\triangle}$ is defined to be the flag completion of $\Gamma_{n}^{\triangle}$.   
There is a natural choice of a subcomplex \(Z \subset \Gamma_{n}^{\triangle}\) such that the pair \((X_{r,d}, Y)\) can be replaced by \((\Gamma_{n}^{\triangle}, Z)\). Let $K$ be the subcomplex of $\Gamma_{n}^{\triangle}$ spanned by the set $H\cup v(H)$ where $v(H)$ is the cone vertex corresponding to the coset $H$, and take $Z$ to be the union of all $G$-translates of $K$, then, every connected component of $Z$ is contractible. Proposition \ref{acyclic} and Proposition \ref{eilenbergmclane} remain true if one replaces  \((X_{r,d}, Y)\) by \((\Gamma_{n}^{\triangle}, Z)\). So, one can prove Theorem \ref{maintheorem1}  using \((\Gamma_{n}^{\triangle}, Z)\) as well. The argument needed to prove acyclicity (or contractibility) of $X_{r,d}$ is rather different than that of $\Gamma_{n}^{\triangle}$ as it does not make use of tools such as dismantlability.  
\end{remark} 

We now give a proof of Corollary \ref{cd_corollary}. 

\begin{proof}[Proof of Corollary \ref{cd_corollary}:]
 Let $G$ be a torsion-free group such that $G$ is relatively hyperbolic with respect to a subgroup $H$ of finite cohomological dimension. By Theorem \ref{maintheorem1}, the relative cohomological dimension $\textrm{cd}_{R}(G,H)$ is finite. Let $m=\textrm{max}\{ \textrm{cd}_{R}(G,H) , \textrm{cd}_{R}(H)\}$. Then, both $H^{i}(G,H;M)$ and $H^{i}(H;M)$ vanish for all $i>m$ and for all $RG$-modules $M$. It follows from Proposition \ref{LES_pair} that $H^{i}(G,M)$ vanishes for all $i>m$ and all $RG$-modules $M$. In other words, $\textrm{cd}_{R}(G)\leq m$. 
\end{proof}
We now provide a proof of Corollary \ref{main_corollary}. We start by proving a couple of lemmas which will be used in the proof of Corollary \ref{main_corollary}. 
\begin{lemma}\label{LES}
    For any exact sequence $0\rightarrow M'\xrightarrow{i} M\xrightarrow{j} M''\rightarrow 0$ of $RG$-modules, and for any $n$, there is a natural map $\partial: H^{n}(G,H;M'')\rightarrow H^{n-1}(G,H;M')$ such that the sequence 
    \newline
    \begin{multline}
                 0 \rightarrow H^{0}(G,H;M')\rightarrow H^{0}(G,H;M)\rightarrow H^{0}(G,H;M'')\\\xrightarrow{\partial} H^{1}(G,H;M')\rightarrow H^{1}(G,H;M)\rightarrow H^{1}(G,H;M'') \rightarrow \dots 
    \end{multline}
    is exact. 
\end{lemma}

\begin{proof}
Let $\dots \rightarrow P_{1}\rightarrow P_{0}\rightarrow\Delta_{G/H}$ denote a free resolution of $\Delta_{G/H}$. For each $P_{i}$, the modules $ \textrm{Hom}_{G}(P_{i},M)$ fit into an exact sequence,
$$0\rightarrow \textrm{Hom}_{G}(P_{i},M')\xrightarrow{i_{\ast}} \textrm{Hom}_{G}(P_{i},M)\xrightarrow{j_{\ast}} \textrm{Hom}_{G}(P_{i},M'')\rightarrow 0.$$ 
This gives rise to a short exact sequence of chain complexes. The existence of $\partial$ and the associated long exact sequence now follow from a standard result in homological algebra (cf. \cite[Prop. 1.3.1]{Weibel_1994}). 
\end{proof}

\begin{lemma}\label{free_module}
   If $\textrm{cd}_{R}(G,H)<\infty$, then
   $$\textrm{cd}_{R}(G,H) = \{ n \mid H^{n}(G,H;F) \neq 0 \text{ for some free } G\text{-module } F \}.$$ 
\end{lemma}

\begin{proof}
    Let $M$ be a $RG$-module such that $H^{n}(G,H;M)\neq 0$ for $n=\textrm{cd}_{R}(G,H)$. Let $F$ be a free module such that $F$ admits a surjective $RG$-module morphism onto $M$. Let $K$ denote the kernel of this morphism. 
    
    Consider the long exact sequence obtained by applying Lemma \ref{LES} to the exact sequence $0\rightarrow K\rightarrow F \rightarrow M \rightarrow  0 $ :
    $$\dots\rightarrow H^{n}(G,H;K)\rightarrow H^{n}(G,H;F)\rightarrow H^{n}(G,H;M) \rightarrow 0.$$ 
    Since $ H^{n}(G,H;M)\neq 0$, it follows that $ H^{n}(G,H;F)\neq 0$.   
\end{proof}

We will also need the following result in order to prove Corollary \ref{main_corollary}:

\begin{lemma}\label{commute}\cite[Prop. 4.6]{brown2012cohomology}
 If a $RG$-module $M$ admits a finite free resolution such that each term is a finitely generated $RG$-module, then $\textrm{Ext}_{G}^{\ast}(M,\cdot)$ commutes with direct limits.    
\end{lemma}

\begin{proof}[Proof of Corollary \ref{main_corollary}:]
    Following Lemma \ref{free_module}, let $F$ be a free $RG$-module such that 
    $$\cd_{R}(G,H) = \sup \{ i \mid H^{i}(G,H;F)\neq 0 \}.$$ 
    Since $F$ is a free $RG$-module, it is isomorphic to a direct sum $\bigoplus_{i\in I} RG$ of copies of $\mathbb{Z}G$. By Lemma \ref{commute}, 
    $$H^{i}(G,H;F)\cong \bigoplus _{i\in I} H^{i}(G,H;RG).$$
    Thus, $H^{i}(G,H;F)\neq 0$ if and only if $H^{i}(G,H;RG)\neq 0$. 
\end{proof}

\section{Application to QI classification of relatively hyperbolic groups}\label{section_4}
In this Section we give a proof of Theorem \ref{qi_theorem} using the results proven in the previous section.

\begin{proposition}
    Let $G$ and $G'$ be two groups such that $G$ is hyperbolic with respect to a subgroup $H < G$ and $G'$ is relatively hyperbolic with respect to a subgroup $H'$. Let $q: G' \to G$ be a quasi-isometry of groups. Assume that both $H$ and $H'$ are unconstricted. Then, the image of any left coset $gH$ under $q$ is within a bounded Hausdorff distance of some coset $g'H'$ of $H'$.
    
    In other words, $q$ is a quasi-isometry between pairs $(G,H)$ and $(G',H')$.  
\end{proposition}

\begin{proof}
    Let $S$ be a symmetric generating set for $G$, and let $d_{S}$ denote the corresponding word metric on $G$. Similarly, let $S'$ be a symmetric generating set for $G'$, and let $d_{S'}$ be the corresponding word metric on $G'$. Let $L \geq 1$ and $C > 0$ be such that $q: (G',d_{S'}) \to (G,d_{S})$ is an $(L,C)$ quasi-isometry. Let $\bar{q}: G \to G'$ denote the quasi-inverse of $q$.  
    
    Let $\mathcal{H}$ denote the collection $\{gH \mid g \in G\}$ of all left cosets of $H$. Then, $G$ is asymptotically tree-graded with respect to $\mathcal{H}$. It follows from \cite[Corollary 5.8]{DRUTU2005959} that the image $q(g'H')$ of any coset is contained in the $R$-neighborhood $N_{R}(gH)$ of some coset $gH \in \mathcal{H}$ for some constant $R$ that depends only on $L$, $C$, $G$, and $S$. Similarly, the image $\bar{q}(gH)$ is contained in the $R'$-neighborhood of some coset $g''H'$.  
    
    As $d_{S'}(x, \bar{q} \circ q(x)) \leq D$ for all $x \in G'$, we have that the infinite set $\bar{q} \circ q(g'H')$ is contained in the set $N_{M}(g'H') \cap N_{M}(g''H')$, where $M = \max\{D, R'\}$. Since $G'$ is asymptotically tree-graded with respect to $\mathcal{H'}$, the collection of all left cosets of $H'$, the diameter of $N_{r}(g_{1}H') \cap N_{r}(g_{2}H')$ is uniformly bounded for any two distinct cosets $g_{1}H'$ and $g_{2}H'$ of $H'$ \cite[Theorem 4.1]{DRUTU2005959}. It follows that $g'H' = g''H'$. 
    
    Now, let $y \in gH$. Then, there exists a point $x \in g'H'$ such that $d_{S}(\bar{q}(y), x) \leq R'$. Thus,
    \[
        d_{S}(q(x), y) \leq d_{S}(q \circ \bar{q}(y), y) + d_{S}(q \circ \bar{q}(y), q(x)) \leq D + L R' + C.
    \]
    Thus, \( gH \subset N_{D+LR'+C}(q(g'H')) \). It follows that the Hausdorff distance between $q(g'H')$ and $gH$ is finite. 
\end{proof}

\begin{proof}[Proof of Theorem \ref{qi_theorem}:]
    By Proposition 4.1 and Theorem \ref{groff}, the Bowditch boundaries \( \partial_{B}(G,H) \) and \( \partial_{B}(G',H') \) are homeomorphic. By Theorem \ref{manningwang}, the relative cohomology \( H^{i}(G,H;RG) \) is isomorphic to the \v{C}ech cohomology of the Bowditch boundary \( \check{H}^{i-1}(\partial_{B}(G,H);R)\). This implies that the two cohomology groups \( H^{i}(G,H;RG) \) and \( H^{i}(G',H';RG') \) are isomorphic for each $i \geq 0$. The result now follows from Corollary \ref{main_corollary}.
\end{proof}
\section{Examples}\label{section_examples}
In this section, we give examples of relatively hyperbolic pairs $(G,H)$ for which $\textrm{cd}_{\mathbb{Z}}(G,H)$ can be computed explicitly.

\begin{example}\label{multi-ended}
Let $H$ be a finitely generated group, and let $H'$ be a hyperbolic group of cohomological dimension $n$, where $n > 2$. Let $G = H \ast H'$. Then $G$ is relatively hyperbolic with respect to $H$. 

Let $Y$ be a $K(H,1)$ space, and let $Y'$ be an $n$-dimensional $K(H',1)$ space. Let $X = Y \vee Y'$ be the wedge sum of $Y$ and $Y'$. Then $\pi_{1}(X) = G$, and the inclusion $Y \hookrightarrow X$ induces the inclusion $H \hookrightarrow G$. Thus, $(X,Y)$ is an Eilenberg–MacLane pair for $(G,H)$. 

Applying Lemma~\ref{classifyingpairlemma1} to the pair $(X,Y)$, we get that $\cd_{R}(G,H) \leq n$.
\end{example}

\begin{example}\label{manifold}
Let $G$ denote the fundamental group of a cusped hyperbolic $n$-manifold with one cusp, and let $H$ denote the cusp subgroup. Then $G$ is relatively hyperbolic with respect to $H$. 

Furthermore, $(G,H)$ forms a relative $PD(n)$ pair \cite[Theorem 6.3]{BIERI1978277}; that is,
\begin{equation}
H^{i}(G,H,\mathbb{Z}G) \cong 
    \begin{cases*}
      \mathbb{Z} & if $i=n$, \\
      0          & otherwise.
    \end{cases*}
\end{equation}
In particular, $\cd_{\mathbb{Z}}(G,H) = n$.
\end{example}

Let $G$ be the fundamental group of a cusped hyperbolic 3-manifold with a single cusp, and let $H \cong \mathbb{Z}^{2}$ be the cusp subgroup. Then $\cd(G,H) = 3$. Let $G' = \mathbb{Z}^2\ast \pi_{1}(S_{g})$, where $\pi_{1}(S_{g})$ denotes the fundamental group of a closed surface of genus $g \geq 2$. Then $G'$ is relatively hyperbolic with respect to the $\mathbb{Z}^{2}$ subgroup, and by Example~\ref{multi-ended}, $\cd_{R}(G', \mathbb{Z}^{2}) = 2$. Note that both $G$ and $G'$ have (absolute) cohomological dimension equal to 2. Thus, one can use Theorem~\ref{qi_theorem} to show that $G$ and $G'$ are not quasi-isometric. 

$G'$ is multi-ended, while $G$ is one-ended. Hence, applying Theorem~\ref{qi_theorem} is somewhat overkill. It is instructive to look at an example of a one-ended group that is relatively hyperbolic with respect to a single $\mathbb{Z}^{2}$ subgroup and has relative cohomological dimension equal to 2.

\begin{example}\label{jsj}
Let $H_{1} = \mathbb{Z}^{2} = \langle a, b \mid [a, b] \rangle$, and let $H' = \pi_{1}(S_{2})$ be the fundamental group of a closed genus-2 surface endowed with a hyperbolic metric. Let $c$ be a simple closed geodesic curve, and let $\gamma$ be the corresponding element in $H'$. Let $G_{1}$ denote the amalgamated product $H_{1} \ast_{\mathbb{Z}} H'$ obtained by identifying the subgroups $\langle a \rangle$ and $\langle \gamma \rangle$. Then $G_{1}$ is relatively hyperbolic with respect to $H_{1}$.

A relative classifying space for the pair $(G_{1}, H_{1})$ can be constructed as follows. Let $Y$ denote the 2-torus, and let $Y'$ denote the genus-2 surface. Let $X$ denote the space obtained by gluing the meridian of $Y$ to the curve $c$. Then $\pi_{1}(X) \cong G_{1}$. Both $Y$ and $Y'$ admit a CAT(0) metric, and the gluing can be performed in such a way that the subspaces are convex and the gluing map is an isometry. As a result, $X$ can be endowed with a CAT(0) metric; in particular, its universal cover is contractible. Thus, $(X,Y)$ is an Eilenberg–MacLane pair for the pair $(G_{1}, H_{1})$. Using Lemmas~\ref{classifyingpairlemma2} and \ref{classifyingpairlemma1}, we have that $\cd_{\mathbb{Z}}(G_{1}, H_{1}) = 2$.
\end{example}

Thus, one can use Theorem~\ref{qi_theorem} to distinguish between $G$ and $G_{1}$. Note that $G_{1}$ splits as an amalgam over the cyclic subgroup $\mathbb{Z} = \langle \gamma \rangle$, whereas $G$ does not split over any of its 2-ended subgroups. It is known that splittings over 2-ended subgroups are preserved under quasi-isometries \cite{papasogluwhyte}. 

We now give another example of a group that is relatively hyperbolic with respect to a single $\mathbb{Z}^{2}$ subgroup and has relative cohomological dimension 2. The example below does not appear to have an obvious $\mathbb{Z}$-splitting, although we do not know of a way to determine whether it admits a $\mathbb{Z}$-splitting.

\begin{example}
Let $H_{2} = \mathbb{Z}^{2} = \langle a, b \mid [a, b] \rangle$, and let $H'$ be the genus-2 surface group. Let $F = H \ast H'$ be their free product. Choose $r \in F$ such that $r$ is not conjugate into either $H$ or $H'$, and let $G_{2} = F/\langle\langle r \rangle\rangle$. Then the quotient map $q: F \to G_{2}$ restricts to a monomorphism on both $H$ and $H'$. Furthermore, $r$ can be chosen so that $G_{2}$ is one-ended and satisfies the $C'_{*}(1/8)$ small-cancellation condition \cite[Section 4.4]{Cordes2016RelativelyHG}, which implies that $G_{2}$ is relatively hyperbolic with respect to $H_{2}$. 

Using a result of Howie \cite{howie}, we show that $\cd_{\mathbb{Z}}(G_{2}, H_{2}) = 2$. Note that both $H_{2}$ and $H'$ are locally indicable \cite[pg.~421, Theorem 1]{howie}. Let $\mathbb{T}^{2}$ denote the 2-torus, and let $S_{2}$ be a copy of the closed genus-2 surface. Let $X' = \mathbb{T}^{2} \vee S_{2}$, so that $\pi_{1}(X') \simeq F$. Let $\Phi_{r} : S^{1} \to X'$ be a continuous map along a closed path corresponding to the word $r$. Let $X$ be the CW-complex obtained by attaching a 2-cell along $\Phi_{r}$: that is, 
\[ X = \mathbb{D}^{2} \sqcup X'/(x \sim \Phi_{r}(x)). \] 
By the Van Kampen theorem, $\pi_{1}(X) = G_{2}$. By \cite[pg.~421, Theorem 1]{howie}, $X$ is aspherical. Thus, $(X, \mathbb{T}^{2})$ is an Eilenberg–MacLane pair for $(G_{2}, H_{2})$. By Lemma~\ref{classifyingpairlemma2}, $\cd_{\mathbb{Z}}(G_{2}, H_{2}) \leq 2$. It is not difficult to see that $H^{2}(X, \mathbb{T}^{2}) \neq 0$, and consequently, $\cd_{R}(G_{2}, H_{2}) = 2$.
\end{example}

\begin{example}
For any two elements $x, y \in G$ of a group $G$ and $m \in \mathbb{N}$, define $(x, y)_{m}$ to be the word of length $m$ given by
\[
(x,y)_{m} =
    \begin{cases*}
      (xy)^{\frac{m-1}{2}} \cdot x & if $m$ is odd, \\
      (xy)^{\frac{m}{2}}           & if $m$ is even.
    \end{cases*}
\]
Any two-generator Artin group, given by the presentation $\langle x, y \mid (x,y)_{m} = (y,x)_{m} \rangle$, is unconstricted \cite[Example 10.1]{behrstockdrutumosher}. In this final example, we show how to construct, for each $n = 4k, k \in \mathbb{N}$, a group $G_{n}$ such that:
\begin{enumerate}
    \item $G_{n}$ is relatively hyperbolic with respect to a subgroup $H_{n}$ that is isomorphic to a two-generator Artin group;
    \item $G_{n}$ has property (T);
    \item $\cd_{\mathbb{Z}}(G_{n}, H_{n}) = n$.
\end{enumerate}
Property (T) implies that $G_{n}$ does not split as an amalgam or HNN extension.

Let $n=4k$ and let $\Gamma_{n}$ be a hyperbolic group of cohomological dimension $n$ with property (T). For example, one can take $\Gamma_{n}$ to be a cocompact lattice in $\textrm{Sp}(k,1)$ for $n=4k$. Let $K$ be a quasi-convex free subgroup of $\Gamma_{n}$ of rank 2, and let $a, b \in K$ such that $a$ and $b$ freely generate $K$. Then $\Gamma_{n}$ is relatively hyperbolic with respect to $K$.
\\
\textit{Claim 1: } $\cd(\Gamma_{n}, K) = n$.
\begin{proof}
Since $\cd(\Gamma_{n}) = n$, we have $H^{n}(\Gamma_{n}; \mathbb{Z}\Gamma_{n}) \neq 0$. As $K$ is a free group, $\cd(K) = 1$, so $H^{i}(K; \mathbb{Z}\Gamma_{n}) = 0$ for all $i > 1$. Using the long exact sequence in Proposition~\ref{LES_pair}, we obtain $H^{i}(\Gamma_{n}, K; \mathbb{Z}\Gamma_{n}) \simeq H^{i}(\Gamma_{n}; \mathbb{Z}\Gamma_{n})$ for all $i > 2$. The result now follows by Theorem~\ref{maintheorem1}.
\end{proof}

Let $F$ be a finite subset of $\Gamma_{n}$ such that the conclusion of Theorem~\ref{dehn_filling} holds. We make the following claim:

\textit{Claim 2: } For sufficiently large $m$, $F \cap \langle\langle (a,b)_{m} \cdot (b,a)_{m}^{-1} \rangle\rangle = \emptyset$.
\begin{proof}
Let $S = \{a, b\}$ and let $|\cdot|_{S}$ denote word length in $K$ with respect to $S$. Let $D$ be the diameter of $F \cap K$. It suffices to show that for large enough $m$, any $g \in \langle\langle r \rangle\rangle$ has word length greater than $D$. This follows from the inequalities
\[
|grg^{-1}|_{S} \geq |r|_{S} \geq 4m - 2.
\]
\end{proof}

Let $N = \langle\langle r \rangle\rangle_{K}$ be the normal closure of $\langle r \rangle$ in $K$, and let $\langle\langle N \rangle\rangle$ be its normal closure in $\Gamma_{n}$. Claim 2 shows that $\langle\langle r \rangle\rangle$ satisfies the assumptions of Theorem~\ref{dehn_filling}. It follows that $G_{n} = \Gamma_{n}/\langle\langle N \rangle\rangle$ is relatively hyperbolic with respect to $H_{n} = K/N$. Since $G_{n}$ is a quotient of a group with property (T), it also has property (T).

\textit{Claim 3: } $\cd(G_{n}, H_{n}) = n$.
\begin{proof}
By Theorem~\ref{binsun}, we have:
\[
H^{i}(G_{n}, H_{n}; \mathbb{Z}G_{n}) \cong H^{i}(\Gamma_{n}, K; \mathbb{Z}G_{n}).
\]
Since $\cd(\Gamma_{n}, K) = n$, the right-hand side vanishes for all $i > n$. By Corollary~\ref{main_corollary}, it suffices to show that $H^{n}(G_{n}, H_{n}; \mathbb{Z}G_{n}) \neq 0$. Let $q: \mathbb{Z}G_{n} \to \mathbb{Z}$ be the augmentation map sending each generator to 1, and let $M = \ker q$. Using the long exact sequence from Lemma~\ref{LES}, we get a surjective morphism:
\[
q: H^{n}(\Gamma_{n}, K; \mathbb{Z}G_{n}) \to H^{n}(\Gamma_{n}, K; \mathbb{Z}).
\]
Since $H^{n}(\Gamma_{n}, K; \mathbb{Z}) \neq 0$, it follows that $H^{n}(G_{n}, H_{n}; \mathbb{Z}G_{n}) \neq 0$.
\end{proof}
\end{example}

\section{Multiple Peripherals}
The analogous statement of Theorem \ref{maintheorem1} holds and the the proof has only
superficial changes. 
\begin{theorem} Let $R$ be a ring.
Let $G$ be a group without $R$-torsion and $\mathcal{P}=\{P_{1},\dots, P_{n}\}$ be a finite collection of subgroups of $G$ such that $G$ is relatively hyperbolic with respect to $\mathcal{P}$. Then, the $RG$-module $\Delta_{G/\mathcal{P}}^{R}$ is a module of type $FP$ over $\mathbb{Z}G$. As a result, $\cd_R(G,\mathcal{P})$ is finite. 
\end{theorem}
The analogous statement for Theorem \ref{qi_theorem} involves a few subtleties. Suppose we are given two quasi-isometric groups $G$ and $G'$ such that $G$ is relatively hyperbolic with respect to a finite collection of unconstricted subgroups $\mathcal{P}$. Using results from \cite{behrstockdrutumosher} one obtains a finite collection $ \mathcal{P}'$ of $G'$ such that each subgroup in $\mathcal{P}$ is quasi-isometric to some subgroup in $\mathcal{P}'$. Moreover, the quasi-isometry $q$ between $G$ and $G'$ induces a quasi-isometry of pairs $(G,\mathcal{P})$ and $(G',\mathcal{P'})$. If one \textit{a priori} knows a collection $\mathcal{Q}$ of unconstricted subgroups with respect to which $G'$ is relatively hyperbolic, we need to ensure that $\mathcal{P}'$ and $\mathcal{Q}$ are the same (up to quasi-isometry). This ambiguity can be resolved as follows: Consider the quasi-isometry $i:G'\rightarrow G'$ induced by the identity map. Since both collections $\mathcal{P}'$ and $\mathcal{Q}$ are unconstricted, $i$ has to induce a quasi-isometry between pairs $(G',\mathcal{P}')$ and $(G,\mathcal{Q})$. As a result, $\partial_{B}(G',\mathcal{P}')$ is homeomorphic to $\partial_{B}(G',\mathcal{Q})$.  We are now ready to state the correct analogue of Theorem \ref{qi_theorem}. 
\begin{theorem}
    Let $G$ and $G'$ be quasi-isometric groups that do not contain $R$-torsion, and suppose that $G$ is relatively hyperbolic with respect to a collection $\mathcal{P}$ of unconstricted subgroups of type $F_{\infty}$. Then, there exists a collection of type $F_{\infty}$, unconstricted subgroups $\mathcal{P}'$ of $G'$ such that the pairs $G'$ is relatively hyperbolic with respect to $G'$. Moreover, for any finite  collection $\mathcal{Q}$  of subgroups of $G'$ that are unconstricted, type $F_{\infty}$ and such that $G'$ is relatively hyperbolic with respect to $\mathcal{Q}$,  we have $\cd(G,\mathcal{P})=\cd(G',\mathcal{Q})$. 
\end{theorem}

\section*{Acknowledgments}
The author would like to thank David Hume and John Mackay for several helpful conversations. The author would like to thank the anonymous referee for their comments and suggestions on a previous draft of this paper.  
\bibliographystyle{alpha}
\bibliography{main}

\textsc{School of Mathematics, University of Bristol, Bristol, BS8 1UG}

\textit{E-mail address}: \texttt{cr22307@bristol.ac.uk}

\end{document}